\documentclass{amsart}
\newtheorem{thr}{Theorem}[section]
\newtheorem{lem}[thr]{Lemma}

\newtheorem{pr}[thr]{Proposition}

\theoremstyle{definition}

\newtheorem{conv}[thr]{Convention}
\newtheorem{notat}[thr]{Notation}
\newtheorem{prob}[thr]{Problem}

\theoremstyle{remark}
\newtheorem{remr}[thr]{Remark}

\numberwithin{equation}{section}

\def\R{\mathbb{R}}

\def\V{{\mathcal V}}

\def\R{\mathbb{R}}

\begin{document}

\title[An upper bound for nonnegative rank]{An upper bound for nonnegative rank}

\author{Yaroslav Shitov}
\address{National Research University Higher School of Economics, 20 Myasnitskaya Ulitsa, Moscow 101000, Russia}
\email{yaroslav-shitov@yandex.ru}


\begin{abstract}
We provide a nontrivial upper bound for the nonnegative rank of rank-three matrices, which allows us to prove that
$\left\lceil\frac{6n}{7}\right\rceil$ linear inequalities suffice to describe a convex $n$-gon up to a linear projection.
\end{abstract}

\maketitle

\section{Preliminaries}

Consider a convex polytope $P\subset\R^n$.
An \textit{extension}~\cite{FRT, GG} of $P$ is a polytope $Q\subset\R^d$ such that $P$
can be obtained from $Q$ as an image under a linear projection from $\R^d$ to $\R^n$.
An \textit{extended formulation}~\cite{GG, Kai} of $P$ is a description of $Q$ by linear equations and
linear inequalities (together with the projection). The \textit{size}~\cite{GG, Kai} of the extended formulation is
the number of facets of $Q$. The \textit{extension complexity}~\cite{GG, Kai} of a polytope $P$ is the smallest size of any extended
formulation of $P$, that is, the minimal possible number of inequalities in the description of $Q$.
The number of facets of $Q$ can sometimes be significantly smaller~\cite{FRT} than that of $P$,
and this phenomenon can be used to reduce the complexity of linear programming problems
useful for numerous applications~\cite{CCZ, FRT, Kai}.

An important result providing the linear algebraic characterization of extended formulations
has been obtained in 1991 by Yannakakis~\cite{Yan}. Let a polytope $P$ (with $v$ vertices and $f$ facets) be defined as the
set of all points $x\in\R^n$ satisfying the conditions $c_i(x)\geq \beta_i$ and $c_j(x)=\beta_j$, for
$i\in\{1,\dots,f\}$ and $j\in\{f+1,\dots,q\}$, where $c_1,\dots,c_q$ are linear functionals on $\R^n$.
A slack matrix $S=S(P)$ of $P$ is an $f$-by-$v$ matrix satisfying $S_{it}=c_i(p_t)-\beta_i$, where $p_1,\dots,p_v$
denote the vertices of $P$, and we note that $S$ is nonnegative. The following well-known result
(see~\cite[Corollary 5]{GG} and also~\cite[Lemma 3.1]{GRT}) characterizes the rank of $S(P)$
in terms of the dimension of $P$.

\begin{pr}\label{statrankslac}
A slack matrix of a polytope $P$ has classical rank one greater than the dimension of $P$.
\end{pr}

The result by Yannakakis points out the connection between extension complexity and nonnegative
factorizations and can now be formulated as follows~\cite{GG, Kai, Yan}.

\begin{thr}\cite[Theorem 2]{Kai}\label{thrYan}
The extension complexity of a polytope $P$ is equal to the minimal $k$ for which $S(P)$ can be written
as a product of $f$-by-$k$ and $k$-by-$v$ nonnegative matrices.
\end{thr}

In general, the smallest integer $k$ for which there exists a factorization $A=BC$ with $B\in\R_+^{n\times k}$
and $C\in\R_+^{k\times m}$ is called the \textit{nonnegative rank} of a nonnegative matrix $A\in\R_+^{n\times m}$.
Nonnegative factorizations are being widely studied and used in data analysis, statistics, computational biology,
clustering and numerous other applications~\cite{CR}. There are still many open questions on nonnegative rank interesting
for different applications, and a considerable part of them is related to providing the bounds on the nonnegative
rank in terms of other matrix invariants~\cite{FMPTdW, GG, Kai}.

In fact, it has still been unknown whether any nontrivial upper bound for the nonnegative rank exists in terms
of the classical rank function. It is easy to show that the nonnegative rank of a matrix equals~\cite{CR} the
classical rank if one of them is less than $3$. However, even for a rank-three $m$-by-$n$ matrix,
no upper bound (instead of $\min\{m,n\}$, which is trivial) for the nonnegative rank has been known.

\begin{prob}\cite[Conjecture 3.2]{BL}\label{probBL}
Assume $n\geq3$. Does there exist a rank-three $n$-by-$n$ nonnegative matrix with nonnegative rank equal to $n$?
\end{prob}

In view of Proposition~\ref{statrankslac} and Theorem~\ref{thrYan}, one can ask a related question on whether
there exists a convex $n$-gon with extension complexity equal to $n$, for every $n$. For $n\leq 5$, Problem~\ref{probBL}
has been solved in the positive in~\cite{GG}. In~\cite{GPT} it was noted that a sufficiently irregular convex hexagon
has full extension complexity, stating the positive answer for $n=6$. For $n\geq7$, the problem has been open.

Lin and Chu~\cite{LC} claimed a positive answer for Problem~\ref{probBL}, but their argument has been
shown to contain a gap~\cite{GG, Hr}. A negative answer for Problem~\ref{probBL} has been obtained
in~\cite{GG} for a special case of so-called Euclidean distance matrices. The factorizations of
those matrices have been studied subsequently in~\cite{Hr}, and the logarithmic upper bounds have
been obtained in a number of important special cases. A detailed investigation of extended formulations
of convex polygons has been undertaken in~\cite{FRT}, but the question about an $n$-gon with extension
complexity equal to $n$ has also been left open.

In our paper we solve Problem~\ref{probBL} and prove that for $n>6$, the answer is negative.
In fact, we provide a nontrivial upper bound for the nonnegative rank of matrices in terms
of classical rank and prove that an $m$-by-$n$ rank-three matrix cannot have nonnegative rank
greater than $\left\lceil\frac{6\min\{m,n\}}{7}\right\rceil$. We also answer the question
on extension complexity and show that a convex $n$-gon has extension complexity at most
$\left\lceil\frac{6n}{7}\right\rceil$. That is, we prove that any convex $n$-gon admits
a description with $\left\lceil\frac{6n}{7}\right\rceil$ linear inequalities up to a projection.

The organization of the paper is as follows. In the second section, we prove the main result in a
special case of slack matrices of convex heptagons, thus showing that any convex heptagon admits a
description with six linear inequalities. In the third section, we use those results and prove the
main results of our paper, which include the upper bounds for the extension complexity of a polygon
and for the nonnegative rank of a rank-three matrix.

\section{Factoring a slack matrix of a convex heptagon}

In this section, we will prove that slack matrices of convex heptagons have nonnegative ranks less than $7$.
The considerations of this section deal with matrices having not more than seven rows and seven columns,
and we adopt the following convention in order to make the presentation more concise.

\begin{conv}\label{convmod7}
Throughout this section, the row and column indexes of the matrices considered
are to be understood as the elements of the ring $\mathbb{Z}/7\mathbb{Z}$. In
particular, $A_{3+6,6+1}$ will stand for the $(2,7)$th entry of a matrix $A$.
Also, we will use the letters $i$ and $j$ only for denoting such indexes in
the present section, and we operate with $i$ and $j$ as with elements from
$\mathbb{Z}/7\mathbb{Z}$, throughout the section.
\end{conv}

Let us introduce a certain special form of matrices which will be important for the considerations of the present section.
By $W[i,j,k]$ we denote the submatrix of $W$ formed by the rows with indexes $i$, $j$, and $k$.

\begin{notat}\label{convnotat}
Given a real vector $\alpha=(a_1,a_2,a_3,b_1,b_2,b_3)$. By $W(\alpha)$ we will denote the $7$-by-$3$ matrix
$$\left(
  \begin{array}{ccccccc}
    0 & 0 & 1 & 1 & a_1 & a_2 & a_3 \\
    1 & 0 & 0 & 1 & 1 & 1 & 1 \\
    1 & 1 & 0 & 0 & b_1 & b_2 & b_3 \\
  \end{array}
\right)^\top,$$
and by $\V(\alpha)$ the $7$-by-$7$ matrix with $(i,j)$th entry equal to $\det W[i-1,j-2,j-1]$.
\end{notat}

The following lemma points out a symmetry in the construction of $\V$.

\begin{lem}\label{lemforwlog}
Matrices $\V(a_1,a_2,a_3,b_1,b_2,b_3)$ and $\V(b_3,b_2,b_1,a_3,a_2,a_1)$
coincide up to relabeling the rows and columns.
\end{lem}

\begin{proof}
Perform the permutation $(16)(25)(34)$ on the row indexes and $(17)(26)(35)$ on the column indexes of $\V(b_3,b_2,b_1,a_3,a_2,a_1)$.
\end{proof}

Let us present a useful special case when the nonnegative rank of $\V$ is not full.

\begin{lem}\label{lemàfactors}
Given a real vector $\psi=(a_1,a_2,a_3,b_1,b_2,b_3)$ for which
the matrix $V=\V(\psi)$ satisfies $V_{ij}>0$ if $i\notin\{j-1,j\}$.
If $a_1+b_1\geq a_2+b_2$ and $a_3+b_3\geq a_2+b_2$, then $V$ has nonnegative rank less than $7$.
\end{lem}

\begin{proof}
One can check that $V=FG$, where
$$F=\left(\begin{array}{cccccc}
0& 0& 1& V_{41} + V_{47}& V_{61}& 0\\
0& 0& 0& 1& a_1 - a_2 + b_1 - b_2& 1\\
V_{31}& 0& 0& 1& V_{37}&0\\
V_{41}& 1& 0& 0& V_{47}& 0\\
-a_2 + a_3 - b_2 + b_3& 1& 0&0& 0& 1\\
V_{61}& V_{31} + V_{37}& 1& 0& 0& 0\\
0&V_{31}& 1& V_{47}& 0& 0
\end{array}\right),$$
$$G=\left(\begin{array}{ccccccc}
1& V_{32}/V_{31}& 0& 0& 0& 0& 0\\
0& V_{21}/V_{31}& 1& 0& 0& 0& 0\\
0& 0& V_{13}& 1& V_{65}& 0& 0\\
0& 0& 0& 0& 1& V_{57}/V_{47}& 0\\
0& 0& 0& 0& 0& V_{65}/V_{47}& 1\\
V_{72}& 0& 0& 1& 0& 0&V_{57}
\end{array}\right).$$
\end{proof}

Now we show how can one construct new full-rank matrices from given.

\begin{lem}\label{convnotat2}
Given a real vector $\psi=(a_1,a_2,a_3,b_1,b_2,b_3)$ for which
the matrix $V=\V(\psi)$ satisfies $V_{ij}>0$ if $i\notin\{j-1,j\}$.
Take $\alpha_1=(1 - a_3 - b_3)/(1 - b_3)$, $\alpha_2=(a_1 - a_1b_3 - a_3 + a_3b_1)/(a_1 - a_1b_3)$,
$\alpha_3=(a_2 - a_2b_3 - a_3 + a_3b_2)/(a_2 - a_2b_3)$, $\beta_1=a_3$, $\beta_2=a_3/a_1$, $\beta_3=a_3/a_2$.
Then the matrix $U=\V\left(\alpha_1,\alpha_2,\alpha_3,\beta_1,\beta_2,\beta_3\right)$ satisfies $U_{ij}>0$
if $i\notin\{j,j+1\}$ and has nonnegative rank equal to that of $V$.
\end{lem}

\begin{proof}
One can check that $V=Q_1UQ_2$, where
$$Q_1=\left(
  \begin{array}{ccccccc}
0& 1& 0& 0& 0& 0& 0\\
0& 0& 1& 0& 0& 0& 0\\
0& 0& 0& 1/(1 - b_3)&0& 0& 0\\
0& 0& 0& 0& 1/a_3& 0& 0\\
0& 0& 0& 0& 0& 1/a_3& 0\\
0& 0&0& 0& 0& 0& a_1/a_3\\
a_2/a_3& 0& 0& 0& 0& 0& 0
\end{array}
\right),$$
$$Q_2=\left(
\begin{array}{ccccccc}
0& 0& 0& 0& 0& 0& \frac{a_1 a_2 (1 - b_3)}{a_3}\\
a_2(1-b_3)& 0& 0& 0& 0& 0& 0\\
0& a_3 (1 - b_3)& 0& 0& 0& 0& 0\\
0& 0& a_3& 0& 0& 0& 0\\
0& 0& 0& 1& 0&0& 0\\
0& 0& 0& 0& \frac{1 - b_3}{a_3}& 0& 0\\
0& 0& 0& 0& 0&\frac{a_1 (1 - b_3)}{a_3}& 0
\end{array}
\right).$$
Since the numbers $1-b_3=V_{42}$, $a_1=V_{63}$, $a_2=V_{73}$, and $a_3=V_{13}$ are positive, the result follows.
\end{proof}

The following six real sequences will be important in our considerations.

\begin{notat}\label{lemsame}
Given a real vector $\psi=(a_1,a_2,a_3,b_1,b_2,b_3)$ for which
the matrix $V=\V(\psi)$ satisfies $V_{ij}>0$ if $i\notin\{j-1,j\}$.
We will consider the six sequences $\alpha_1(t)$, $\alpha_2(t)$, $\alpha_3(t)$,
$\beta_1(t)$, $\beta_2(t)$, and $\beta_3(t)$ of reals defined by
$\alpha_1(0)=a_1$, $\alpha_2(0)=a_2$, $\alpha_3(0)=a_3$,
$\beta_1(0)=b_1$, $\beta_2(0)=b_2$, $\beta_3(0)=b_3$, and also
$$\alpha_1(t+1)=\frac{1 - \alpha_3(t) - \beta_3(t)}{1 - \beta_3(t)},$$
$$\alpha_{\chi+1}(t+1)=\frac{\alpha_\chi(t) - \alpha_\chi(t)\beta_3(t) - \alpha_3(t) + \alpha_3(t)\beta_\chi(t)}
{\alpha_\chi(t) - \alpha_\chi(t)\beta_3(t)}\,\,\,\mbox{for}\,\,\,\chi\in\{1,2\},$$
$$\beta_1(t+1)=\alpha_3(t),\,\,\,\beta_2(t+1)=\alpha_3(t)/\alpha_1(t),\,\,\,\beta_3(t+1)=\alpha_3(t)/\alpha_2(t).$$
\end{notat}

\begin{remr}\label{remrwelldefseq}
Lemma~\ref{convnotat2} shows that the sequences  $\alpha_1(t)$, $\alpha_2(t)$,
$\alpha_3(t)$, $\beta_1(t)$, $\beta_2(t)$, and $\beta_3(t)$ are well defined.
\end{remr}

It turns out that the sequences introduced are in fact cyclic.

\begin{lem}\label{lemcyc}
Given a real vector $\psi=(a_1,a_2,a_3,b_1,b_2,b_3)$ for which
the matrix $V=\V(\psi)$ satisfies $V_{ij}>0$ if $i\notin\{j-1,j\}$.
Then $\alpha_1(7)=a_1$, $\alpha_2(7)=a_2$, $\alpha_3(7)=a_3$,
$\beta_1(7)=b_1$, $\beta_2(7)=b_2$, $\beta_3(7)=b_3$.
\end{lem}

\begin{proof}
By routine computation.
\end{proof}

The following lemma gives a necessary condition for a matrix to be full-rank.

\begin{lem}\label{lem+4}
Given a real vector $\psi=(a_1,a_2,a_3,b_1,b_2,b_3)$ for which
the matrix $V=\V(\psi)$ satisfies $V_{ij}>0$ if $i\notin\{j-1,j\}$.
Then $\alpha_1(2)+\beta_1(2)\leq\alpha_2(2)+\beta_2(2)$ implies that
$\alpha_2(6)+\beta_2(6)<\alpha_3(6)+\beta_3(6)$.
\end{lem}

\begin{proof}
A routine computation shows that
$$\alpha_2(2)+\beta_2(2)-\alpha_1(2)-\beta_1(2)=\frac{(-a_3 + a_2 (1 - b_3))\,V_{32}V_{21}}{V_{31}V_{73}V_{42}V_{52}},$$
so the sign of $\alpha_2(2)+\beta_2(2)-\alpha_1(2)-\beta_1(2)$ equals that of $-a_3 + a_2 (1 - b_3)$. Similarly,
$$\alpha_3(6)+\beta_3(6)-\alpha_2(6)-\beta_2(6)=\frac{V_{46}\,(-a_3 + a_1 (1 - b_3))}{V_{15}V_{36}},$$
so the sign of $\alpha_3(6)+\beta_3(6)-\alpha_2(6)-\beta_2(6)$ is that of $-a_3 + a_1 (1 - b_3)$.
It remains to note that $1-b_3=V_{42}>0$ and $a_1-a_2=V_{37}>0$.
\end{proof}

In fact, we can obtain a stronger condition that holds for full-rank matrices.

\begin{lem}\label{lemthrou}
Given a real vector $\psi=(a_1,a_2,a_3,b_1,b_2,b_3)$ for which
the matrix $V=\V(\psi)$ satisfies $V_{ij}>0$ if $i\notin\{j-1,j\}$ and has full nonnegative rank.
Then either $\alpha_1(t)+\beta_1(t)<\alpha_2(t)+\beta_2(t)<\alpha_3(t)+\beta_3(t)$ for every $t$
or $\alpha_1(t)+\beta_1(t)>\alpha_2(t)+\beta_2(t)>\alpha_3(t)+\beta_3(t)$ for every $t$.
\end{lem}

\begin{proof}
Assume that $\alpha_1(t)+\beta_1(t)\leq \alpha_2(t)+\beta_2(t)$, for some $t$. Applying Lemma~\ref{lem+4} to the vector
$\psi'=(\alpha_1(t+5),\alpha_2(t+5),\alpha_3(t+5),\beta_1(t+5),\beta_2(t+5),\beta_3(t+5))$ and taking into account Lemma~\ref{lemcyc},
we obtain that $\alpha_2(t+4)+\beta_2(t+4)<\alpha_3(t+4)+\beta_3(t+4)$. Lemma~\ref{lemàfactors} then shows that
$\alpha_1(t+4)+\beta_1(t+4)<\alpha_2(t+4)+\beta_2(t+4)$, and we conclude that
$\alpha_1(t+4k)+\beta_1(t+4k)<\alpha_2(t+4k)+\beta_2(t+4k)<\alpha_3(t+4k)+\beta_3(t+4k)$, for any positive integer $k$.

Now assume $\alpha_1(t)+\beta_1(t)>\alpha_2(t)+\beta_2(t)$. By Lemma~\ref{lemàfactors}, we have $\alpha_2(t)+\beta_2(t)>\alpha_3(t)+\beta_3(t)$,
and so by Lemma~\ref{lem+4}, $\alpha_1(t+3)+\beta_1(t+3)>\alpha_2(t+3)+\beta_2(t+3)$. Finally, we conclude that
$\alpha_1(t+3k)+\beta_1(t+3k)>\alpha_2(t+3k)+\beta_2(t+3k)>\alpha_3(t+3k)+\beta_3(t+3k)$, for any positive $k$.
\end{proof}

Finally, let us show that a matrix $\V(\psi)$ can not have full nonnegative rank.

\begin{lem}\label{lemnofull}
Given a vector $\psi=(a_1,a_2,a_3,b_1,b_2,b_3)$ for which
the matrix $V=\V(\psi)$ satisfies $V_{ij}>0$ if $i\notin\{j-1,j\}$.
Then $\V$ has nonnegative rank less than $7$.
\end{lem}

\begin{proof}
Assume the converse and apply the results of Lemma~\ref{lemforwlog} and Lemma~\ref{lemthrou}. We can assume without a loss of generality that
$\alpha_1(t)+\beta_1(t)<\alpha_2(t)+\beta_2(t)<\alpha_3(t)+\beta_3(t)$, for any nonnegative integer $t$.
Note that $\alpha_3(0)+\beta_3(0)-\alpha_1(0)-\beta_1(0)=a_3+b_3-a_1-b_1$, and routine computations also allow us to check that
$$\alpha_2(1)+\beta_2(1)-\alpha_1(1)-\beta_1(1)=\frac{V_{13}\,(b_1 + (a_1-1) b_3)}{V_{63} V_{42}},$$
$$\alpha_2(2)+\beta_2(2)-\alpha_1(2)-\beta_1(2)=\frac{V_{13} V_{21}\,(a_2 (1 - b_3)-a_3)}{V_{73} V_{31} V_{42} V_{52}}.$$
Noting that also $1-b_3=V_{42}>0$ and $V_{37}=a_1-a_2>0$, we obtain
\begin{equation}\label{eqsys111}
b_3(1-a_1)<b_1,\,\,\,a_3<a_2 (1 - b_3),\,\,\,a_3+b_3>a_1+b_1,\,\,\,b_3<1,\,\,\,and\,\,\,a_1>a_2.
\end{equation}

Now let us check that~(\ref{eqsys111}) is a contradiction. In fact, the first of these inequalities implies $a_1+b_1>a_1+b_3-b_3a_1$,
taking into an account the third we obtain $a_3+b_3>a_1+b_3-b_3a_1$. Thus we have $a_3>a_1(1-b_3)$, which implies $a_3>a_2(1-b_3)$
because of the last two inequalities.
\end{proof}

Let us now check that $7$-by-$7$ matrices of a more general form have nonnegative rank at most $6$ as well.
By $U[r_1,r_2,r_3|c_1,c_2,c_3]$ we denote the submatrix of $U$ formed by the rows with indexes $r_1$, $r_2$, $r_3$
and columns with $c_1$, $c_2$, $c_3$.

\begin{lem}\label{lem7x7gen}
Assume that a $7$-by-$7$ matrix $U$ has classical rank $3$ and
satisfies $U_{ij}=0$ if $i\in\{j-1,j\}$ and $U_{ij}>0$ otherwise.
Then $U$ has nonnegative rank less than $7$.
\end{lem}

\begin{proof}
Denote by $U'$ the matrix obtained from $U$ by multiplying the third column by $U_{54}/U_{53}$,
the fifth column by $U_{24}/U_{25}$, the third row by $\frac{U_{25}}{U_{24}U_{35}}$, the fourth row by $\frac{U_{53}}{U_{43}U_{54}}$,
the $i'$th row by $1/U_{i'4}$ (for $i'$ from $1,2,5,6,7$). So we have
$$U'=\left(
\begin{array}{ccccccc}
0 & 0 & a_3 & 1 & b_3 & U'_{16} &U'_{17} \\
U'_{21} & 0 & 0 & 1 & 1 &U'_{26} &U'_{27} \\
U'_{31} & U'_{32} & 0 & 0 & 1 & U'_{36} & U'_{37} \\
U'_{41} & U'_{42} & 1 & 0 & 0 & U'_{46} & U'_{47} \\
U'_{51} & U'_{52} & 1 & 1 & 0 & 0 & U'_{57} \\
U'_{61} & U'_{62} & a_1 & 1 & b_1 & 0 & 0 \\
0 & U'_{72} & a_2 & 1 & b_2 & U'_{76} & 0 \\
\end{array}
\right).$$
Since $U'$ has classical rank $3$, there are certain real constants $c_1,\dots,c_7$ such that $U'_{ij}=c_j\,\det U'[i,j-1,j|3,4,5]$, for any $i$ and $j$.
Therefore, we obtain $U'_{ij}=c_j V_{ij}$ for any $i$ and $j$, where $V$ is the matrix $\V(a_1,a_2,a_3,b_1,b_2,b_3)$ from Notation~\ref{convnotat}.
Since $V_{13}=V_{32}$ and $V_{72}=V_{21}$, the numbers $c_1$, $c_2$, and $c_3$ are of the same sign. Similarly, $V_{65}=V_{46}$ and $V_{76}=V_{57}$, so
that the numbers $c_5$, $c_6$, and $c_7$ are of the same sign as well. Further, since $V_{24}=V_{25}=V_{43}=1$, we obtain $c_3=c_4=c_5=1$, and
the numbers $c_1,\dots,c_7$ are thus all positive. So we can conclude that $U$ and $V$ coincide up to multiplying the rows and columns by positive
numbers, and the result then follows from Lemma~\ref{lemnofull}.
\end{proof}

Now we can prove the main result of the present section.

\begin{thr}\label{thrstart}
A\,slack\,matrix\,of\,a\,convex\,heptagon\,has\,nonnegative\,rank\,at\,most\,$6$.
\end{thr}

\begin{proof}
Proposition~\ref{statrankslac} shows that the slack matrix $S$ of a convex heptagon has classical rank equal
to $3$. Therefore, $S$ satisfies the assumptions of Lemma~\ref{lem7x7gen} up to renumbering the rows and columns.
\end{proof}

\section{Main results}

In this section we prove the main results of our paper. Let us start with a corollary of Theorem~\ref{thrstart}
which gives a positive answer for Problem~\ref{probBL} in the case $n=7$.

\begin{thr}\label{thrupnonneg1}
Let $A$ be a nonnegative $7$-by-$n$ matrix with classical rank equal to $3$. Then the nonnegative
rank of $A$ does not exceed $6$.
\end{thr}

\begin{proof}
Consider the standard simplex $\Delta$ consisting of points $(x_1,\ldots,x_7)$ with nonnegative coordinates satisfying $\sum_{i=1}^7 x_i=1$.
Since $\Delta$ contains $7$ facets, the intersection of $\Delta$ with the column space of $A$ is a polygon $I$ with $k$ vertices, and $k\leq 7$.
Form a matrix $S$ of column coordinate vectors of vertices of $I$, then $A=SB$ with $B$ nonnegative. If $k<7$, then the result follows directly
from that $A=SB$, and if $k=7$, then by Theorem~\ref{thrstart}, $S$ has nonnegative rank less than $7$ being a slack matrix for $I$.
\end{proof}

Now we can provide a nontrivial upper bound for the nonnegative rank of matrices with classical rank equal to $3$,
thus providing a negative solution for Problem~\ref{probBL} in the case $n\geq7$.

\begin{thr}\label{thrupbound}
The nonnegative rank of a rank-three matrix $A\in\R_+^{m\times n}$ does not exceed $\left\lceil\frac{6\min\{m,n\}}{7}\right\rceil$.
\end{thr}

\begin{proof}
By Theorem~\ref{thrupnonneg1}, any seven rows of $A$ can be expressed as linear combinations with nonnegative coefficients
of certain six nonnegative rows, so the nonnegative rank of $A$ does not exceed $\left\lceil\frac{6m}{7}\right\rceil$. The
nonnegative rank is invariant under transpositions, so the result follows.
\end{proof}

Together with the result from~\cite{GPT}, where it was noted that a sufficiently irregular convex hexagon
has full extension complexity, Theorem~\ref{thrupbound} provides a full answer for Problem~\ref{probBL}.
Namely, the following result is true.

\begin{thr}\label{thrupbound3}
If $n\geq7$, then the nonnegative rank of any rank-three $m$-by-$n$ nonnegative matrix is less than $n$.
For $k\in\{3,4,5,6\}$, there are $k$-by-$k$ rank-three matrices with nonnegative rank equal to $k$.
\end{thr}

Finally, we can prove an upper bound for the extension complexity of convex polygons.

\begin{thr}\label{threxcom}
The extension complexity of any convex $n$-gon does not exceed $\left\lceil\frac{6n}{7}\right\rceil$.
\end{thr}

\begin{proof}
By Proposition~\ref{statrankslac} and Theorem~\ref{thrupbound}, the nonnegative rank of a slack matrix
does not exceed $\left\lceil\frac{6n}{7}\right\rceil$, so the result follows from Theorem~\ref{thrYan}.
\end{proof}

\bigskip


The author is grateful to the participants of the workshop on
Communication complexity, Linear optimization, and Lower bounds for the nonnegative
rank of matrices held at Schloss Dagstuhl in February, 2013, for enlightening discussions on the topic.


\bigskip


\begin{thebibliography}{99}


\bibitem{BL}
{ L. B. Beasley, T. J. Laffey, } Real rank versus nonnegative rank, \textit{Linear Algebra Appl.}, \textbf{431} (2009), 2330--2335.

\bibitem{CR}
{ J. E. Cohen, U. G. Rothblum, } Nonnegative ranks, decompositions, and factorizations
of nonnegative matrices, \textit{Linear Algebra Appl.}, \textbf{190}(1993), 149--168.

\bibitem{CCZ}
M. Conforti, G. Cornuejols, G. Zambelli, Extended
formulations in combinatorial optimization, \textit{4OR}, \textbf{8(1)} (2010), 1--48.

\bibitem{FMPTdW}
S. Fiorini, S. Massar, S. Pokutta, H. R. Tiwary, R. de Wolf, Linear vs. semidefinite extended formulations: exponential separation and strong lower bounds,
in \textit{Proc. 44th Symp. on Th. of Comp.}, 95--106, 2012, ACM.

\bibitem{FRT}
S. Fiorini, T. Rothvo{\ss}, H. R. Tiwary, Extended formulations for polygons, \textit{Disc. Comp. Geom.}, \textbf{48(3)} (2012), 1-11.

\bibitem{GPT}
J. Gouveia, P. A. Parillo, R. R. Thomas. Lifts of convex sets and cone factorizations. ArXiv preprint arXiv:1111.3164.

\bibitem{GRT}
J. Gouveia, R. Z. Robinson, R. R. Thomas. Polytopes of minimum positive
semidefinite rank. Arxiv preprint arXiv:1205.5306.

\bibitem{GG}
{ N. Gillis, F. Glineur, } On the Geometric Interpretation of the Nonnegative Rank, \textit{Linear Algebra Appl.}, \textbf{437} (2012), 2685--2712.

\bibitem{Hr}
{ P. Hrube\v{s}, } On the nonnegative rank of distance matrices, \textit{Information Processing Letters}, \textbf{112(11)} (2012), 457--461.

\bibitem{Kai}
V. Kaibel, Extended Formulations in Combinatorial Optimization, \textit{Optima} \textbf{85} (2011), 2--7.

\bibitem{LC}
{ M. M. Lin and M. T. Chu, } {On the nonnegative rank of Euclidean distance matrices}, \textit{Linear Algebra Appl.}, \textbf{433} (2010), 681--689.

\bibitem{Yan}
M. Yannakakis, Expressing combinatorial optimization problems by linear programs, \textit{Comput. System Sci.}, \textbf{43} (1991) 441--466.

\end{thebibliography}
\end{document}